\documentclass[12pt]{amsart}

\usepackage{amsmath}
\usepackage{amssymb,amsthm,amscd}

\usepackage{fullpage}
\usepackage[lmargin=3cm,rmargin=3cm,tmargin=2cm,bmargin=2cm, marginpar=2cm ]{geometry}

\usepackage{xcolor}
\usepackage{hyperref}
\hypersetup{
    colorlinks=true,
    linkcolor=blue,
    citecolor=cyan!80!black,
    filecolor=cyan!80!black,      
    urlcolor=cyan!80!black,
} 

\usepackage{extarrows} %

\usepackage[normalem]{ulem} %
\usepackage[toc,page]{appendix}

\usepackage{graphicx}

\usepackage{tikz}
\usepackage{pgfplots}
\usepackage{pgfplotstable}
\usepackage{subcaption}
\pgfplotsset{compat=1.7}
\usepgfplotslibrary{groupplots}
\usetikzlibrary{matrix,arrows,positioning,calc,chains}

\usepackage{enumitem}
\setlist[enumerate,1]{label=(\roman*)}

\def\Aut{\operatorname{Aut}}

\def\SAut{\operatorname{SAut}}

\def\Pic{\operatorname{Pic}}

\def\AffCone{\operatorname{AffCone}}
\def\Ample{\operatorname{Ample}}
\def\Eff{\operatorname{\overline{NE}}}

\def\supp{\operatorname{supp}}

\def\PP{{\mathbb P}}

\def\ZZ{{\mathbb Z}}

\def\RR{{\mathbb R}}
\def\QQ{{\mathbb Q}}
\def\NN{{\mathbb N}}

\def\G{{\mathbb G}}

\def\AA{{\mathbb A}}
\def\PP{{\mathbb P}}

\def\K{{\mathbb K}}

\def\embed{\hookrightarrow}

\def\0{\circ}

\def\isom{\stackrel{\sim}{\longrightarrow} } %

\def\Aut{\operatorname{Aut}}
\def\SAut{\operatorname{SAut}}

\def\Pic{\operatorname{Pic}}

\def\AffCone{\operatorname{AffCone}}
\def\Ample{\operatorname{Ample}}
\def\relint{\operatorname{rel.int.}}
\def\Cone{\operatorname{Cone}}
\def\Eff{\operatorname{\overline{NE}}}

\def\supp{\operatorname{supp}}

\def\Pol{\operatorname{Pol}}
\def\Forb{\operatorname{Forb}}
\def\Subd{\operatorname{Subd}}

\theoremstyle{plain}
\newtheorem{theorem}{Theorem}[section]
\newtheorem{lemma}[theorem]{Lemma}

\newtheorem*{conjecture*}{Conjecture}
\newtheorem{question}[theorem]{Question}
\newtheorem{proposition}[theorem]{Proposition}

\theoremstyle{definition}
\newtheorem{definition}[theorem]{Definition}
\newtheorem{example}[theorem]{Example}
\newtheorem{remark}[theorem]{Remark}

\newtheorem{convention}[theorem]{Convention}

\usepackage[frozencache]{minted} %
\usepackage[T2A]{fontenc}

\begin{document}

\author{Alexander Perepechko}

\address{
  Faculty of Computer Science, HSE University,
  Pokrovsky blvd. 11, Moscow, 109028 Russia
}

\email{a@perep.ru}

\title{Generic flexibility of affine cones over  del Pezzo surfaces in Sagemath}

\thanks{The research was carried out at the HSE University at the expense of the Russian Science Foundation (project no. 21-71-00062).}

\dedicatory{To Yuri Prokhorov on the occasion of his 60th birthday}

\begin{abstract}
Generic flexibility of affine cones over Fano varieties is a subject of active study recently.
For del Pezzo surfaces the question is completely studied in degree at least 3, and partially in degree 2.

We present a Sagemath module that facilitates most operations for verifying the generic flexibility of affine cones over del Pezzo surfaces and weak del Pezzo surfaces of arbitrary degree, depending on a polarization. The combinatorial approach used in this module is based on the formalism of bubble cycles and the colimit of Picard groups of blowups of the projective plane.

As an example, we verify generic flexibility of affine cones over some polarizations of surfaces of degree 1 under certain conditions and over arbitrary very ample polarizations of weak del Pezzo surfaces of degree 6. 
\end{abstract}
\maketitle

\section{Introduction}

Let $Y$ be a (weak) del Pezzo surface of arbitrary degree polarized by a very ample divisor $H$. We study when the affine cone over $Y$ is generically flexible. Equivalently, such a cone has an open subset with an infinitely transitive action of the special automorphism group on it.

We are working over an algebraically closed field $\K$ of characteristic zero. 
Let $X$ be an affine algebraic variety over $\K$.
A point $p\in X$ is called \emph{flexible} if the tangent space $T_pX$ is spanned by the tangent vectors to the orbits of actions of the additive group of the field $\G_a=\G_a(\K)$ on $X$. 
The variety $X$ is called \emph{generically flexible} if it contains a flexible point, and \emph{flexible} if each smooth point of $X$ is flexible, see \cite{AFKKZ},  \cite{Pe20}.

All the $\G_a$-actions on $X$ generate the \emph{special automorphism group} $\SAut X\subset\Aut X$.
 A group $G$ is said to act on a set $S$ \emph{infinitely transitively} if it acts transitively on the set of ordered $m$-tuples of pairwise distinct points in $S$ for any $m\in\NN$.
 The following theorem explains the significance of the flexibility concept.
\begin{theorem}[{\cite[Theorem 2.2]{AFKKZ}}]\label{th:AFKKZ}
  Let $X$ be an affine algebraic variety of dimension $\ge2$. Then the following conditions are equivalent:
  \vspace{-2pt}
  \begin{enumerate}
    \item The variety $X$ contains a flexible point;
    \item the group $\SAut X$ acts on $X$ with an open orbit;
    \item the action of $\SAut X$ on the open orbit is infinitely transitive.
  \end{enumerate}
  Then all points of the open orbit are flexible.
\end{theorem}

Affine cones are a rich source of flexible and generically flexible affine algebraic varieties. Affine cones over polarizations of del Pezzo surfaces are extensively studied in last 10 years.
 The flexibility of affine cones over del Pezzo surfaces of degree $\ge4$ corresponding to any very ample polarization was confirmed in \cite{AKZ}, \cite{PW16}, \cite{P13}.

 The generic flexibility of affine cones over cubic surfaces was confirmed in \cite{Pe20} for polarizations by any very ample divisor $H$ which is not plurianticanonical, i.e., not proportional to the anticanonical one;
 the absence of $\G_a$-actions for plurianticanonical polarizations was proven in \cite[Corollary 1.8]{CPW13}. 
 The case of degree 2 is studied in \cite{KimPark21} and \cite{KimWon25}.
 
 General theory and questions on flexible varieties are presented in \cite{Arzhantsev23}, \cite{Kaliman21}, \cite{Kaliman2022embedding}.
 There are also studies on flexibility of affine cones over singular del Pezzo surfaces \cite{Won22}, \cite{Saw24} and over Fano--Mukai fourfolds \cite{PZ21flex}, \cite{hoff2022flexibility}, \cite{HangTruong2023flex}.

 The computations for del Pezzo surfaces of lower degrees depend heavily on the polarization and involve dealing with families of ample divisors as polyhedral cones of large dimension.
 Cylinders themselves and (weak) del Pezzo surfaces can be described with combinatorial data based on the formalism of bubble cycles and the colimit of Picard groups.
 
 Thus, we present a Sagemath module \cite{Pe23-sage} based on this combinatorial approach that facilitates most operations with said cones and allows verifying of the  generic flexibility of affine cones over del Pezzo surfaces of arbitrary degree.
 As an application, we check generic flexibility in degree 1 for two families of polarizations, see Theorems~\ref{th:deg1-B} and \ref{th:deg1-C}.

 Moreover, we adapted the module for weak del Pezzo surfaces and proved that affine cones over weak del Pezzo surfaces of degree 6 are generically flexible for all very ample polarizations, see Theorem~\ref{th:deg6-weak}.

 Let us outline the contents of the paper.
 In Section~\ref{sec:flex-crit} we recall definitions on cylinder collections and a criterion of generic flexibility for affine cones from \cite{Pe20}.
 In Section~\ref{sec:bubble-cycles} we introduce bubble classes in terms of bubble cycles and the colimit of Picard groups, and study their properties.
 In Section~\ref{sec:bubble-systems} we propose a way to describe combinatorial types of (weak) del Pezzo surfaces and cylinders on them in terms of linear systems and corresponding bubble classes.
In Section~\ref{sec:cones} we recall notation on cones in the Mori cone of a del Pezzo surface $Y$. We also introduce open subdivisions of such cones in Definition~\ref{def:subd} and Proposition~\ref{pr:subd}.
In Section~\ref{sec:cylinders} we present some known constructions of cylinder collections and their description in terms of bubble classes. 
In Section~\ref{sec:deg2} we provide examples of their usage for surfaces of degree $\ge2$ with the corresponding code.
 In Section~\ref{sec:deg1} we provide some new results in degree 1.
 In Section~\ref{sec:weak} we prove generic flexibility for (non-toric) weak del Pezzo surfaces of degree 6.
In Appendix~\ref{sec:usage} we provide a brief introduction on the usage of the module \cite{Pe23-sage}.

\section{Flexibility criterion}\label{sec:flex-crit}
Here we present the flexibility criterion from \cite{Pe20}.
In this section we let $Y$ denote a normal projective variety, $H$ a very ample divisor on $Y$, and $X=\AffCone_H(Y)$ the affine cone over $Y$ corresponding to $H$. By $\pi\colon X\setminus\{0\}\to Y$ we denote the projectivizing map.

\begin{definition}[{\cite[Definition 3.1.7]{KPZ11}}]
An open subset $U\subset Y$ is called $H$-\emph{polar} if
the complement $Y\setminus U$ is the support of some effective $\QQ$-divisor equivalent to $H$ over $\QQ$.
\end{definition}

\begin{definition}[{\cite[Definition 2.3]{Pe20}}]
Let $\mathcal{U}=\{U\}$ be a collection of open affine subsets of $Y$.
\begin{enumerate}
    \item We say that a subset $Z$ of $Y$ is $\mathcal{U}$-\emph{invariant} if for any $U\in\mathcal{U}$ the intersection $U\cap Z$ is $\SAut(U)$-invariant. 
    \item We say that $\mathcal{U}$ is \emph{transversal} if $\bigcup_{U\in \mathcal{U}} U$ does not admit nontrivial $\mathcal{U}$-invariant subsets.
    \item We say that $\mathcal{U}$ is $H$-\emph{complete} if for any (not necessarily effective) $\QQ$-divisor equivalent to $H$ its support is not contained in the complement $Y\setminus\bigcup_{U\in\mathcal{U}} U$.
\end{enumerate}
\end{definition}

\begin{theorem}[{\cite[Theorem 2.4]{Pe20}}]\label{th:flex}
 Let $\mathcal{U}$ be a transversal collection of $H$-polar open affine subsets of a normal projective variety $Y$ for some very ample divisor $H$. 
 Then there exists an $\SAut(X)$-orbit $S$ on the corresponding affine cone $X$ whose image contains $\bigcup_{U\in\mathcal{U}}U$.
  If, moreover, $\mathcal{U}$ is $H$-complete, then $S$ is open in $X$ and contains the open subset $\pi^{-1}(\bigcup_{U\in\mathcal{U}}U)\subset X$.
\end{theorem}

We operate with collections of cylinders, which are defined as follows.
\begin{definition}[{\cite[Definition 3.1.5]{KPZ11}}]
  An open subset $U\subset Y$ is called a \emph{cylinder} if
  $U\cong \AA^1\times Z$ for some smooth variety $Z$ with $\Pic(Z)=0$. 
  By \emph{fibers} of $U$ we mean $\AA^1$-fibers of the projection to $Z$. 
  There is a natural $\G_a$-action on $U$, whose orbits are fibers of $U$.
\end{definition}

\section{Bubble cycles}\label{sec:bubble-cycles}
The notation of this section is based on \cite[Section~7]{Dolg-CAG}.
Let $X$ be a smooth projective surface. We consider the category $\mathcal{B}_X$ induced by the birational morphisms $\phi_Y\colon Y\to X$ of smooth projective surfaces.  Objects are surfaces $Y$, and morphisms are birational morphisms over $X$. That is, $\mathrm{Mor}(Y,Y')$ consists of at most one element, the composition $\phi_{Y'}^{-1}\circ \phi_Y$, provided that it is a morphism. We say that $Y$ \emph{dominates} $Y'$ in this case.

A surface $Y\in\mathcal{B}_X$ defines a subcategory $\mathcal{B}_Y$.
Each such surface $Y$ is obtained as an \emph{iterative blowup} of $X$, 
$$Y\stackrel{\sigma_1}{\to} Y_1\stackrel{\sigma_2}{\to}\ldots\stackrel{\sigma_n}{\to} Y_n=X,$$
 where each $\sigma_i$ is a contraction of a (-1)-curve into a smooth point and $\phi_Y=\sigma_n\circ\cdots\circ\sigma_1$.

We introduce the following extension of the Picard group $\Pic X$ to iterative blowups.
 \begin{definition}
   The \emph{bubble Picard group} of $X$ is the direct limit
   \begin{equation}\label{eq:pic-bb}
    \Pic\mathcal{B}_X=\varinjlim_{Y\in\mathcal{B}_X}\Pic Y 
   \end{equation}
  under injective maps $\sigma^*:\Pic Y \to \Pic Y'$ for  $Y'\stackrel{\sigma}{\to}Y\to X.$  Since maps $\sigma^*$ preserve intersection forms, $\Pic\mathcal{B}_X$ is also endowed with the intersection form.
 \end{definition}

The \emph{bubble space} $X^{bb}$ of $X$ is the set of the closed points of all interative blowups of $X$ modulo natural equivalence relations:
$$ 
X^{bb} = \sqcup_{Y\in\mathcal{B}_X} Y /\{p_1\sim p_2\},
$$
where $p_1\sim p_2$ if there exist neighbourhoods
$p_i\in U_i\subset Y_i$ such that
$\sigma_{Y_2}^{-1}\circ\sigma_{Y_1}:U_1\cong U_2$, see \cite[Definition~7.3.1]{Dolg-CAG}. It is a coproduct in the category of sets.
We say that a point $q$ is \emph{infinitely near} $p$, where $p,q\in X^{bb}$, and write $q\succ p$, if $q\neq p$ and $\sigma(q)=p$ for some morphism $\sigma$ in $\mathcal{B}_X$. Let $k$ be the minimal number of blowups in $p$ and infinitely near $p$ that is needed to obtain $q$. Then we say that $q$ is of \emph{order} $k$ over $p$ and write $q\succ_k p$.

The following definition is based on \cite[cf. Definition~7.3.2]{Dolg-CAG}, except for an additional positivity condition, which we omit.
\begin{definition}
  A \emph{bubble cycle} on  $X$ is a finite formal sum $\eta=\sum_{p\in X^{bb}} \eta(p) p\in\ZZ^{X^{bb}}$ of points of $X^{bb}$ with integer coefficients. Equivalently, $\eta$ is a map $X^{bb}\to\ZZ$ with a finite support.
We let the \emph{support} of a bubble cycle $\eta$, denoted by $\supp\eta$, be the set of points $p\in X^{bb}$ such that $\eta(p)\neq0$ or $\eta(q)\neq0$ for some $q\succ p$. %
We denote the group of bubble cycles on $X^{bb}$ by $\mathcal{Z}(X^{bb})$.%
\end{definition}

\begin{definition}
  A bubble cycle $\eta$ is called a \emph{blowup cycle} if %
  we have $\eta(p)=1$ for any $p\in\supp\eta$.
\end{definition}

\begin{convention}
  Each bubble cycle $\eta$ on $Y\in\mathcal{B}_X$ %
  induces an iterative blowup $\sigma\colon Y_\eta\to Y$ by first blowing up points $p\in\supp\eta$ that are contained in $Y$, then infinitely near ones of order 1, 
   etc. The map $\sigma$ is called the \emph{resolution} of $\eta$. Note that two bubble cycles $\eta,\eta'\in \mathcal{Z}(Y^{bb})\subset  \mathcal{Z}(X^{bb})$ with the same support generate the same surface  $X_\eta=X_{\eta'}$.  

  Sending $\eta$ to $Y_\eta$, we obtain a natural bijection between blowup cycles $\eta$ on $Y$ and surfaces $Y'\in\mathcal{B}_X$ that dominate $Y$. 
\end{convention}

\begin{convention}\label{cn:bb-pic}
  Consider a point $p\in X^{bb}$ that belongs to a surface $Y\in \mathcal{B}_X$. Let $\sigma:\mathrm{Bl}_pY\to Y$ be the blowup of $Y$ at $p$ and $\mathcal{E}_p$ be the divisor class  of the exceptional curve $E$ of $\sigma$. Then we identify $p$ with $\mathcal{E}_p$.

  By additivity this defines a natural embedding  $\mathcal{Z}(X^{bb})\embed \Pic\mathcal{B}_X$. In particular, a bubble cycle $\eta$ is sent to the Picard group $\Pic X_\eta$ of the corresponding iterative blowup.

Note that for any $Y'$ dominating $\mathrm{Bl}_pY$ as above the divisor class $\mathcal{E}_p\in\Pic(Y')$ contains the only effective divisor $D_p$. Given an iterative blowup $\sigma\colon Y_\eta\to Y$ corresponding to a bubble cycle $\eta$ on $Y$, we call the divisors $D_p$ for $p\in\supp\eta$  the \emph{exceptional configurations} of the birational morphism $\sigma$, see \cite[Section~7.3.1]{Dolg-CAG}.
\end{convention}

In fact, the bubble Picard group  combines the concepts of a bubble cycle and a divisor class, as the following lemma shows.
\begin{lemma}
  \begin{enumerate}
    \item We have
    $$\Pic\mathcal{B}_X= \Pic X \oplus \mathcal{Z}(X^{bb}),$$
    where $\mathcal{Z}(X^{bb})$ is considered as a subgroup of $\Pic\mathcal{B}_X$, following Convention~\ref{cn:bb-pic}.
    \item The groups $\Pic \mathcal{B}_X$ and $\Pic \mathcal{B}_Y$ coincide for $Y\in\mathcal{B}_X$.
    \item We have 
    $$\Pic\mathcal{B}_X= \Pic Y \oplus \mathcal{Z}(Y^{bb})$$
    for any $Y\in\mathcal{B}_X.$
  \end{enumerate}
\end{lemma}
\begin{proof}
  Assertion (i)  directly follows from the definition of $\Pic\mathcal{B}_X$ as the inductive limit of Picard groups under blowups of points.
  Assertion (ii) is trivial, since the corresponding direct limits \eqref{eq:pic-bb} coincide.
  Assertion (iii) follows from the previous two.
\end{proof}

This observation allows us to treat elements of $\Pic \mathcal{B}_X$ as pairs, see the following definition.
\begin{definition}
  A \emph{bubble class} $(D,\eta)$ on $Y\in\mathcal{B}_X$ is a pair of a divisor class $D\in\Pic Y$ and a bubble cycle $\eta$ on $Y^{bb}\subset X^{bb}$.
\end{definition}
  The following lemma shows how the bubble class transforms under blowups as a pair.
\begin{lemma}\label{lm:bubble-class-transform}
  \begin{enumerate}
    \item Consider a bubble class $(D,\eta)$ on $Y\in\mathcal{B}_X$ and its representation $(D',\eta')$ on $Y'$ for a blowup $\sigma\colon Y'\to Y$ at a point $p$ with the exceptional curve $E$. Then $D'=D+kE$ and $\eta'=\eta|_{{Y'}^{bb}}=\eta-kp$, where $k=\eta(p)=D'\cdot E$.
    \item Consider a bubble class $(D,\eta)$ on $Y\in\mathcal{B}_X$ and a iterative blowup $\sigma:Y'\to Y$ that blows up all points in $\supp\eta$. Then  we have
    $$(D,\eta)=\left(D+\sum_{p\in\supp\eta} \eta(p)\mathcal{E}_p,0\right)=D+\sum_{p\in\supp\eta} \eta(p)\mathcal{E}_p$$ 
    on $Y'$, where $\mathcal{E}_p$ is the divisor class corresponding to $p$ on $Y'$.
  \end{enumerate}
\end{lemma}
\begin{proof}
  Note that $\sigma^*(D)=D$, since  $\Pic Y\subset \Pic Y'$ as subgroups of $\Pic\mathcal{B}_X$ by our convention. Then the first assertion is straightforward. The second one is obtained by induction on the number of blowups.
\end{proof}

Each bubble class also defines a linear system on $X$ with a prescribed base scheme, which motivates the following definition.

\begin{definition}
  A \emph{bubble linear system} corresponding to the bubble class $(D,-\eta)$ on $Y\in\mathcal{B}_X$ %
  is the linear system $|D-\eta|$ in notation of \cite[Section~7.3.2]{Dolg-CAG}. %
  
  That is, $|D-\eta|$ consists of effective divisors $D'\in|D|$ that (up to taking a proper transform) have multiplicity at least 
  $\eta(p)$ at $p$ for each $p\in\supp\eta$.
\end{definition}

\begin{remark}
A divisor  $D'\in|D|$ on $Y$ belongs to $|D-\eta|$ if and only if its proper transform $\tilde D$ on $Y_\eta$ satisfies $\tilde D\cdot \mathcal{E}_p\ge\eta(p)$ for any $p\in\supp\eta$.
\end{remark}

Bubble linear systems on $Y\in\mathcal{B}_X$ are exactly complete linear systems on surfaces that dominate $Y$, see the following proposition. 
\begin{proposition}
  Let $L=|D-\eta|$ be a bubble linear system on $Y\in\mathcal{B}_X$ and $\sigma_\eta\colon Y_\eta\to Y$ be a resolution of $\eta$. Then $\sigma_\eta^*(L)$ is the complete linear system $|D'|$, where $D'\in\Pic Y_\eta$ equals $(D,-\eta)$.
  
  Vice versa, let $Y,Y'\in\mathcal{B}_X$ be such that $Y'$ dominates $Y$ and let  $D'\in\Pic Y'$. Then the image of the complete linear system $|D'|$ equals $|D-\eta|$ for the bubble class $(D,-\eta)$ on $Y$ such that $D'=(D,-\eta)$.
\end{proposition}
\begin{proof}
  Let $\sigma:Y_\eta\to\eta$ be an iterative blowup that blows up all points in $\supp\eta$. 
  Then $|D-\eta|$ equals the direct image of $|\sigma^*(D)-\sum_{p\in\supp\eta} \eta(p)\mathcal{E}_p|$, cf. Lemma~\ref{lm:bubble-class-transform}.
     The converse statement is clear.
\end{proof}

\section{Bubble linear systems on $\PP^2$}
\label{sec:bubble-systems}
In this section we let $X=\PP^2$ and restrict to the full subcategory $\mathcal{B}_{dP}\subset\mathcal{B}_{\PP^2}$ of del Pezzo and weak del Pezzo surfaces (except $\PP^1\times\PP^1$).  A \emph{negative curve} on $Y\in\mathcal{B}_{dP}$ is a curve $C\subset Y$ isomorphic to $\PP^1$ with a negative self-intersection index $C^2$. If $Y$ is a del Pezzo surface, then $C^2=-1$, otherwise $C^2$ equals either $-1$ or $-2$.
The notation of this section is based on \cite[Section~8]{Dolg-CAG}.

We consider surfaces in $\mathcal{B}_{dP}$ up to the following combinatorial equivalence, cf. Weyl equivalence in \cite{Lubbes}.

\begin{definition}
  We say that $Y,Y'\in\mathcal{B}_{dP}$ are \emph{equivalent} if there is an isomorphism $\Pic Y\isom\Pic Y'$ respecting the intersection form that induces a bijection of classes of negative curves.
  The equivalence class of surfaces in $\mathcal{B}_{dP}$ is called the \emph{type} of surfaces. It is defined by $\Pic Y$ as a lattice with an intersection form and the subset of classes of negative curves.
  We define the \emph{type} of a blowup cycle $\eta$ on $\PP^2$ as one of $X_\eta$. 
\end{definition}
\begin{remark}
  The subset of $(-2)$-curves on $Y$ defines singularities types on the singular del Pezzo surface $Y'$, which is a (pluri)anticanonical model of $Y$, see \cite[Section~8.1]{Dolg-CAG}. 
\end{remark}

Consider a blowup cycle $\eta$ on $\PP^2$.
By \cite[Corollary~8.1.17]{Dolg-CAG}, the surface $X_\eta$ belongs to $\mathcal{B}_{dP}$ if and only if the following conditions on $\supp\eta$ hold.
\begin{enumerate}
  \item $|\supp\eta|\le8$;
  \item for each $p\in\supp\eta$ there is at most one $q\in\supp\eta$ such that $q\succ_1p$;
  \item no four points lie on a line, i.e., $|L-p_1-\ldots-p_4|=\emptyset$ for any $p_1,\ldots,p_4\in\supp\eta$.
  \item no seven points lie on a conic, i.e.,
  $|2L-p_1-\ldots-p_7|=\emptyset$ for any $p_1,\ldots,p_7\in\supp\eta$.
\end{enumerate}

\begin{convention}
  The type of a blowup cycle $\eta$ for $X_\eta\in\mathcal{B}_{dP}$ is defined by the following possible dependencies:
\begin{itemize}
  \item $p\prec q$ for $p,q\in\supp\eta$;
  \item three points in $\supp\eta$ lie on a line, i.e., $|L-p_1-p_2-p_3|$ is nonempty for some $p_1,p_2,p_3\in\supp\eta$;
  \item five points in  $\supp\eta$ lie on a conic, i.e., $|2L-p_1-\ldots-p_5|$ is nonempty for some $p_1,\ldots,p_5\in\supp\eta$;
  \item all eight points in $\supp\eta$ lie on a cubic, one of them is a node, i.e., $|3L-2p_1-p_2-\ldots-p_8|\neq\emptyset$ for some enumeration choice of $\{p_1,\ldots,p_8\}=\supp\eta$.
\end{itemize}
The set of such dependencies corresponding to $\eta$ we denote by $\text{type}(\eta)$.
\end{convention}

A cylinder $U\subset\mathrm{Bl}_\eta\PP^2$ is described by the following data:
\begin{enumerate}
  \item the pencil of curves $P$ induced by the $\AA^1$-fibration of $U$, which is a 1-dimensional linear subsystem in $|D-\eta'|$ with the base scheme corresponding to $\eta'$ and rational irreducible general fibers that contain $\AA^1$ outside of the only base point or fixed component;
  \item fibers of $P$ in the complement of $U$, including all reducible ones, which are again described in terms of bubble classes.
\end{enumerate}
   So, the cylinder $U$ can be expressed in terms of the combinatorial type of $\eta+\eta'$ and similar data for components of  reducible fibers. %
   Our aim is to extend cylinder constructions to (weak) del Pezzo surfaces of various types in a uniform and algorithmic way.

\section{Picard group and polyhedral cones}\label{sec:cones}

For the rest of this paper we use the following notation.
We denote by 
\begin{itemize}
    \item $Y$ a del Pezzo surface of degree $d$, where $d\in\{1,\ldots,5\}$;
    \item  $\sigma$ a contraction $Y\to\PP^2$;
    \item  $E_1,\ldots,E_m$ the (-1)-curves contracted by $\sigma$ into $p_1,\ldots,p_m\in\PP^2$ respectively, where $m=9-d$; 
    \item  $L_{i,j}$ a line passing through $p_i, p_j$ for given $i,j\in\{1,\ldots,m\}$, $i\neq j$, as well as its proper transform in $Y$;%
    \item  $L$ the class of the proper transform of a general line on $\PP^2$;
    \item  $-K$ the anticanonical divisor class on $Y$;
    \item  $\Eff(Y)$ the Mori cone of numerically effective $\RR$-divisors in $\Pic(Y)\otimes\RR$; it is generated by all the (-1)-curves in $Y$;
    \item  $\Ample(Y)$ the ample cone; it contains all the ample divisors in its interior and is dual to $\Eff(Y)$.
\end{itemize}
Given a $\QQ$-divisor $D$, we denote its divisor class by the same letter.
Note that 
\begin{itemize}
    \item $\Ample(Y)\subset\Eff(Y)$,
    \item $-K \equiv 3L-\sum_{i=1}^mE_i$,
    \item $L_{i,j}\equiv L-E_i-E_j$,
    \item among points $p_1,\ldots,p_m$, no three lie on the same line, no six belong to the same conic, no eight belong to a cubic curve with a node at one of them, see \cite[Theorem IV.2.5]{Manin}.
\end{itemize}

\begin{definition}\label{def:subd}
  Consider a strictly convex rational polyhedral cone $C\subset\QQ^n$ and a vector $r\in C$.
  Let $\mathcal{F}(C,r)$  be the set of all proper faces of $C$ such that the cone $\Cone(F,r)$ generated by $F$ and $r$ does not lie in a proper face of $C$.
  Then we call an \emph{open subdivision} 
  \[
    \Subd(C,r) = \{\Cone(F,r)\mid F\in\mathcal{F}(C,r)\}
  \] 
  the collection of such cones. We call elements $\Cone(F,r)$ \emph{subdivision cones}.
\end{definition}

\begin{proposition}\label{pr:subd}
  Consider an open subdivision $S=\Subd(C,r)$ of a strictly convex rational polyhedral cone $C\subset\QQ^n$ along a ray $r$ in $C$. Then 
  \[
    \relint(C)=\bigsqcup_{C'\in S}\relint(C').
  \]
\end{proposition}
\begin{proof}
  For any $a\in \relint(C)$ not proportial to $r$ we consider a halfplane $H$ generated by $r,-r,a$.
  It intersects $C\setminus\relint(C)$ by some ray $\QQ_{\ge0}b$. 
  Then there is a unique face $F$ such that $b\in\relint(F)$.

  Finally, if $a=r\in\relint(C)$, then the corresponding face $F$ is zero and $r\in\relint\Cone(r)$.
\end{proof}

See \cite[Proposition 4.3]{Pe20} and \cite[Section 2.1]{CPW15} for the following known statements. \begin{proposition}
  Consider the open subdivision $S_1=\Subd(\Eff(Y),-K)$ of $\Eff(Y)$ along $-K$. Any element of $S_1$ up to a choice of $\sigma$ and a permutation of $E_1,\ldots,E_m$ equals one of the cones $B_0,\ldots,B_m, B_P, C$, where
  \begin{itemize}
    \item $B_k=\Cone(-K,E_1,\ldots,E_k)$ for $k=0,\ldots,m$;
    \item $B_P=\Cone(-K,E_1,\ldots,E_{m-2},L_{m-1,m})$;
    \item $C=\Cone(E_1,\ldots,E_{m-1},L_{1,m},\ldots,L_{m-1,m})$.
  \end{itemize}    
\end{proposition}

\begin{proposition}
  Consider the open subdivision $S_2=\Subd(C,L-E_m)$ of $C$ along the fiber class $L-E_m$. Any element of $S_2$, if it contains ample divisor classes, up to a choice of $\sigma$ and a permutation of $E_1,\ldots,E_m$ equals one of the cones $C_0,\ldots,C_m,C_P$, where
  \begin{itemize}
    \item $C_k=\Cone(-K,E_1,\ldots,E_{k},L-E_{m})$ for $k=0,\ldots,m-1$;
    \item $C_P=\Cone(-K,E_1,\ldots,E_{m-2},L_{m-1,m},L-E_m)$.
  \end{itemize}    
\end{proposition}
Proposition~\ref{pr:subd} guarantees that any ample divisor class up to a choice of $E_1,\ldots,E_m$ belongs to the relative interior of one of cones $B_0,\ldots,B_m,B_P,C_0,\ldots,C_P$. Note that notation corresponds to one in \cite{Pe20}, except $B_P$ and $C_P$, which are denoted by $B_5^\prime$ and $C_5^\prime$ in \cite{Pe20}.

\begin{remark}
  Another open subdivision implicitly used, e.g., in \cite[Section 6.3]{CPW15} is as follows. Given a cone $C_k$ or $B_k$, we subdivide it along $E_1+\ldots+E_k$ and up to the permutation of $E_1,\ldots,E_k$ consider only the subdivision cones with non-increasing coordinates in $E_1,\ldots,E_k$.
\end{remark}

The conditions on a divisor $H$ that provide $H$-polarity and $H$-completeness of a cylinder collection are expressed through the polarity and forbidden cones respectively.

\begin{definition}[{\cite[Definition 6.1]{Pe20}}]
  Given an open subset $U\subset Y$, which complement consists of components $D_1,\ldots,D_n$, we define the \emph{polarity cone}
  $\Pol(U)=\Cone(D_1,\ldots,D_n)$. Then $U$ is $H$-polar for any divisor $H$ from $\relint(\Pol(U))$.
  
  Given a collection $\mathcal{U}$ 
  of open subsets of $Y$, we define the \emph{polarity cone} via 
  \[\Pol(\mathcal{U})=\bigcap_{U\in\mathcal{U}}\Pol(U).\] Let $D_1,\ldots,D_n$ be the irreducible curves in $Y\setminus \bigcup_{U\in\mathcal{U}}U$. We define the \emph{forbidden cone} via 
  \[\Forb(\mathcal{U})=\Cone(D_1,\ldots,D_n).\]
  Then $\mathcal{U}$ is $H$-complete for any $H\notin\Forb(\mathcal{U})$.
  \end{definition}
  
\section{Cylinder constructions}\label{sec:cylinders}
Here we recall some cylinder constructions across surfaces of varying degree and discuss their properties in the spirit of Section~\ref{sec:bubble-systems}. More constructions and their implementation in Sagemath are available at \cite{Pe23-sage}, see Appendix~\ref{sec:usage} for usage information. The known coverings implying generic flexibility for different degrees and divisor types are presented as tests in \cite{Pe23-sage}.

For any curve $C\in\PP^2$ we denote its proper transform in $Y$ by $\tilde{C}$. 

\begin{example}\label{ex:lines}
  Consider the pencil of lines in $\PP^2$ passing through a point $p_i$, see \cite[Example 4.1.1]{CPW15}. Then the complement $\PP^2\setminus\bigcup_{j\neq i}L_{ij}$ is a cylinder. The bubble class of the general fiber is $(L,-p_i)$, the base locus is $p_i$, and fibers in the complement lie in bubble classes $(L,-p_i-p_j)$.
   Its polarity and forbidden cones in $\Pic_\QQ Y$ coincide and are generated by $E_1,\ldots,E_m$ and $L-E_i-E_j$ for $j\neq i$. 
\end{example}

Consider a cubic del Pezzo surface $Y$, which is the blowup of $\PP^2$ at points $p_1,\ldots,p_6$, and choose a point $p\in\PP^2$ such that $p,p_1,\ldots,p_5$ lie on a conic. Consider a conic $Q\in|2L-p_1-\ldots-p_5|$ and a line $T\in|L-p-p_6|$ tangent to $Q$.

The cylinder $U$ is the isomorphic preimage 
of $\PP^2\setminus T\cup Q$. Fibration is given by the pencil $\langle Q,2T\rangle$, the corresponding bubble class is $(2L,-p)$.

\begin{example}\label{ex:tangent}
  Consider a conic $Q$ on $\PP^2$ and its tangent line $T$ at $p\in Q$. Then the complement to $Q\cup T$ is a cylinder with a general fiber of bubble class $(2L,p)$. It corresponds to the pencil of curves $P=\langle Q,2T\rangle$.
  
  The moduli space of such pencils is $5$-dimensional, and we may impose at most 5 independent conditions of form `points $p_i$ and $p_j$ belong to the same fiber' or `a point $p_i$ lies on $T$'. For example, let $p_1,\ldots,p_5$ lie on the same fiber, which we denote by $Q$, and $p_6$ on $T$.

  Then the cylinder $U$ on $Y$ is the complement of reducible fibers of $\sigma^*(P)$, which are given by bubble classes $\tilde{Q}=|2L-p_1-\ldots-p_5|$ for $Q$, $|L-p_6|$ for $T$ and $|2L-p_i|$ for $i>6$, if any.
  The polarity cone of $U$ is generated by curves the corresponding divisor classes and classes of exceptional curves $E_1,\ldots,E_m$.

  In this case a pair of such cylinders exists (with same polarity cone) and comprises a transversal collection. It can be seen from the tangency condition on $T$ and $Q$, which is of degree 2.
    Indeed, there are two tangents $T, T'$ to $Q$ through $p_6$. 
      Its forbidden cone is  
      $\Cone(\tilde{Q},E_1,\ldots,E_m)$. Since $E_m$ is non-positive on this cone, it does not contain ample divisors. 
      So, by Theorem~\ref{th:flex}, this collection implies generic transitivity for any ample divisor class in the relative interior of the polarity cone.
      See \cite[Example 4.1.2]{CPW15} and \cite[Example 5.2]{Pe20}. 
  \end{example}

\begin{example}\label{ex:cuspidal}
We follow \cite[Theorem 6.2.2, Case 2.1]{CPW15} and \cite[Example 4.1.13]{CPW15}.
Consider a del Pezzo surface $Y$ of degree $2\le d \le 5$. %
Choose a point $p\in \PP^2$ distinct from $p_1,\ldots,p_m$ and consider a cuspidal curve $C$, a conic $Q$ and lines $L_1,\ldots,L_4$ as follows.
\begin{align*}
    C\in&(3L,-2p-p_1-\ldots-p_m)=(-K_Y,-2p),\\
    Q\in&(2L,-p-p_1-\ldots-p_4),\\
    L_i\in&(L,-p-p_i).
\end{align*}
Then $U=Y\setminus(C\cup Q\cup L_1\cup\ldots\cup L_4)$ is a cylinder corresponding  to  the pencil $\langle 2C, Q+L_1+\ldots+L_4\rangle$ and the bubble class $(6L,-2p-p_1-\ldots-p_4)$ on $\PP^2$, which is $(6L-2E_1-\ldots-2E_4,-2p)$ on $Y$.

Let us consider $\mathrm{Bl}_pY$ and 
contract it to $\PP^2$ by blowing down proper transforms $Q',L_1',\ldots,L_4', E_5,\ldots,E_m$, which are disjoint (-1)-curves on $Y'$, and  let $p_0', \ldots, p_5'$ together with $p_6, \ldots, p_m$ be their images. In other words, we send $\mathrm{Bl}_pY\in\mathcal{B}_{\PP^2}$ isomorphically to some $Y'\in\mathcal{B}_{\PP^2}$ so that $Q',L_1',\ldots,L_4', E_5,\ldots,E_m$ are sent to exceptional curves over $\PP^2$. This isomorphism, say $\phi$, is described by a transformation of the type of the surface induced by the change of contraction to $\PP^2$.

Then $\phi_*(C)\in|L|$, $\phi_*(E_p)\in |2L-p_1'-\ldots-p_5'|$, and $(\sigma_{Y'}\circ\phi)_*(C)$ is a line tangent to the conic $(\sigma_{Y'}\circ\phi)_*(E_p)$, so this example is a variation of Example~\ref{ex:tangent}.

If we start this construction with another point $p'\neq p$ that is also the cusp of an anticanonical cuspidal curve $C'$ on $Y$, then the obtained cylinder $U'$ shares the same data with $U$.

  The resulting pair of cylinders is transversal, its forbidden cone $\Cone(E_{5},\ldots,E_{m})$ does not contain ample divisors and its polarity cone is
\[
  \Cone(-K,2L-E_{1}-\ldots-E_{4}, L-E_{1}, \ldots, L-E_4, E_{5},\ldots,E_{m}).
\]

\end{example}

\begin{remark}
  Let $Y$ be a del Pezzo surface of degree 2, then there is a double cover $Y\to\PP^2$ ramified over a quartic curve $C$ in $\PP^2$, e.g., see \cite[Section 6.3.3]{Dolg-CAG}. It is known that for a generic del Pezzo surface of degree 2 the quartic $C$ does not have hyperinflection points. Indeed, according to \cite[Section 6.4]{Dolg-CAG} and \cite{Cohen-quartic}, the Salmon invariant of degree 60 vanishes on the locus of quartics with an
  inflection bitangent.

  In particular, a generic del Pezzo surface of degree 2  and any del Pezzo surface of degree $\ge3$ contain a cuspidal rational curve in its anticanonical class. Moreover, it contains at least a pair of such curves, e.g., see \cite[Lemma 3.1]{KimPark21}.
\end{remark}

\section{Generic flexibility in degree $\ge2$}\label{sec:deg2}
We provide an example of usage of the introduced constructions.
Other similar statements on generic flexibility in different degrees, e.g., from 
 \cite{P13}, \cite{CPW13}, \cite{Pe20}, \cite{KimPark21}, can be found in \cite{Pe23-sage}.

\begin{example}\label{ex:cuspcubic-flex}
  Let $Y$ be a del Pezzo surface of degree $d\in\{2,\ldots,5\}$ that contains an anticanonical cuspidal rational curve. 
  Let $C$ be one of cones $B_{5-d}$ or $C_{5-d}$. Then Example~\ref{ex:cuspidal} provides an $H$-complete and $H$-polar transversal collection for any $H\in\relint(C).$

  For $d=2$ this corresponds to \cite[Theorem 6.2.3, Case 2.1]{CPW15} and \cite[Case 3.1.2]{KimPark21}.

\end{example}

\begin{example}\label{ex:B2-section}
  Consider the case of degree $2$ and the cone $B_2$. The existence of $\G_a$-actions and generic flexibility are checked in see \cite[Theorem 6.2.3]{CPW15} and \cite[Lemma 5.1]{KimPark21} respectively for the part of the cone $B_2$ as depicted in Fig.\ref{fig:B2_CPW}.
  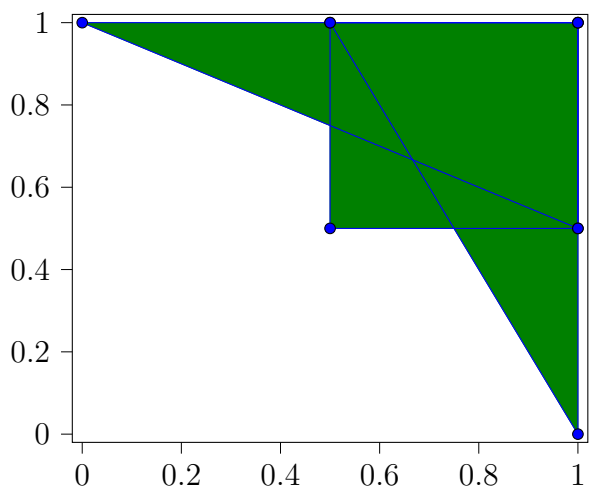
\begin{figure}[!ht]
    \centering
    \begin{tikzpicture}

\definecolor{darkgray176}{RGB}{176,176,176}
\definecolor{green}{RGB}{0,128,0}

\begin{axis}[
tick align=outside,
tick pos=left,
x grid style={darkgray176},
xmin=-0.02, xmax=1.02,
xtick style={color=black},
y grid style={darkgray176},
ymin=-0.02, ymax=1.02,
ytick style={color=black}
]
\path [draw=green, fill=green, line width=0pt]
(axis cs:1,0.5)
--(axis cs:0,1)
--(axis cs:1,1)
--cycle;
\path [draw=green, fill=green, line width=0pt]
(axis cs:0.5,1)
--(axis cs:1,0)
--(axis cs:1,1)
--cycle;
\path [draw=green, fill=green, line width=0pt]
(axis cs:0.5,0.5)
--(axis cs:1,0.5)
--(axis cs:1,1)
--(axis cs:0.5,1)
--cycle;
\addplot [blue, opacity=1.0]
table {%
0 1
1 1
};
\addplot [blue, opacity=1.0]
table {%
0 1
1 0.5
};
\addplot [blue, opacity=1.0]
table {%
1 1
1 0.5
};
\addplot [blue, opacity=1.0]
table {%
1 0
1 1
};
\addplot [blue, opacity=1.0]
table {%
1 0
0.5 1
};
\addplot [blue, opacity=1.0]
table {%
1 1
0.5 1
};
\addplot [blue, opacity=1.0]
table {%
1 1
0.5 1
};
\addplot [blue, opacity=1.0]
table {%
1 1
1 0.5
};
\addplot [blue, opacity=1.0]
table {%
0.5 0.5
0.5 1
};
\addplot [blue, opacity=1.0]
table {%
0.5 0.5
1 0.5
};
\addplot [fill=blue, mark=*, only marks]
table{%
x  y
0 1
1 1
1 0.5
};
\addplot [fill=blue, mark=*, only marks]
table{%
x  y
1 0
1 1
0.5 1
};
\addplot [fill=blue, mark=*, only marks]
table{%
x  y
1 1
0.5 0.5
0.5 1
1 0.5
};
\end{axis}

\end{tikzpicture}
    \caption{The cone $B_2$ in degree 2 covered by polarity cones.
    A point $(a_1,a_2)$ corresponds to 
     the divisor class $H=-K+a_1E_1+a_2E_2$.
     The depicted square with vertices $(0,0),(1,0),(0,1),(1,1)$ is the section of the ample part of $B_2$ by the hyperplane $-K+\bigoplus \QQ E_i$.
     The filled part is covered by polarity cones.}
    \label{fig:B2_CPW}
  \end{figure}

\end{example}

\begin{remark}
  Another approach for  studying partial coverings of introduced subdivision cones is to compute the volume of the covered part relative to the volume of the whole cone. By the volume of the cone we mean the volume of its section with some hyperplane, say $(-K,\cdot)=1$. %
  One could also further subdivide the studied cone.
\end{remark}

\section{Generic flexibility in degree 1}\label{sec:deg1}
Here we present new results on generic flexibility in degree 1.
In the following theorem we use the construction from \cite[Theorem 6.3.1]{CPW15}.
\begin{theorem}\label{th:deg1-B}
  Let $Y$ be a del Pezzo surface of degree $1$ and $H$ be a very ample divisor such that
  \begin{equation}\label{eq:deg1-B-cone}    
  H\in\relint\Cone(2L-E_1,2L-E_2,L-E_3, 2L-E_4-\ldots-E_8,E_1,\ldots,E_8).
\end{equation}
Then the cone $\AffCone_H(Y)$ is generically flexible.
\end{theorem}
\begin{proof}
  It is enough to consider the construction given in Example~\ref{ex:tangent}(1) for $k=3$.
\end{proof}

\begin{remark}
  The condition \eqref{eq:deg1-B-cone} may be represented in other way.
  The intersection of cone in \eqref{eq:deg1-B-cone} with $B_k$, where $k\ge3$ and $k\le7$, equals
  \[
    \Cone(-K+E_1+E_2, -K+E_1+\frac{E_2}{2}+E_3 , -K+\frac{E_1}{2}+E_2+E_3 ,E_1,\ldots,E_k)
  \]
  For $k=8$ the intersection is bigger than that.

  For some $\mu\in\QQ_{>0}$ $H\sim-K+\sum_{i=1}^8 a_iE_i$, where $a_i\ge0$. Then $H$ satisfies \eqref{eq:deg1-B-cone} if and only if
  \[
    2a_1+2a_2+a_3>2,\qquad a_1,a_2,a_3>0.
  \]
  This is the condition in \cite[Theorem 6.3.1]{CPW15}.
\end{remark}

\begin{theorem}\label{th:deg1-C}
  Let $H\in\relint C$ a very ample divisor, where
  \[
    C=\Cone(L-E_1,\ldots,L-E_6, L-E_7-E_8, E_1,\ldots,E_8).
  \]
  Then $\AffCone_H(Y)$ is generically flexible.
\end{theorem}
\begin{proof}
  Consider a collection of a pair of cylinders corresponding to $i=7,8$ in Example~\ref{ex:lines}. 
  Then it is transversal, $C$ is its polarity cone, and the forbidden cone equals $\Cone(L-E_7-E_8, E_1,\ldots,E_8)$, which does not contain ample divisor classes. The statement follows from Example~\ref{th:flex}.
\end{proof}

In degree $\ge2$ the anticanonical class $-K$ is contained in the closure of the polarity cone for some cylinder collection. 
It is not clear whether this holds in degree $1$.

\begin{question}
  Given a del Pezzo surface $Y$ of degree 1,
what is the angle distance between $-K$ and the subset of divisor classes $H$ such that $\AffCone_H(Y)$ is generically flexible?
\end{question}

\section{Weak del Pezzo surfaces}\label{sec:weak}
Let us consider some examples of cylinders on weak del Pezzo surfaces.

\begin{example}\label{ex:weak-deg6-collinear}
  Let $Y$ be a weak del Pezzo surface of degree 6 with a single (-2)-curve of bubble class $(L,-p_1-p_2-p_3)$. Then 
  \[\Ample(Y)=\Cone(L, L-E_1, L-E_2, L-E_3).\]
  Consider a point $p\in\PP^2$ not on the line $\overline{p_1p_2p_3}$ and lines $L_i=\overline{pp_i}\subset\PP^2$.

  Then $\PP^2\setminus \bigcup L_i$ is a cylinder, as well as its preimage in $Y$, which has fibers of bubble class $(L,-p)$.
  The collection of such cylinders with varying $p$ is transversal, $H$-complete on $\Ample(Y)$ and $H$-polar on $\Ample^\circ(Y)$. 
\end{example}

\begin{example}\label{ex:weak-deg6-infinitesimal}
  Let $Y$ be a weak del Pezzo surface of degree 6 with a single (-2)-curve corresponding to the dependency $p_2\succ p_1$. Then 
  \[\Ample(Y)=\Cone(L, L-E_1, L-E_1-E_2, L-E_3,2L-E_1-E_2-E_3).\]
  Consider a point $p\in\PP^2$ on the line $\overline{p_1p_2p_3}$ and some other line $L'\subset\PP^2$ of the bubble class $(L,-p)$. %

  Then $\PP^2\setminus \overline{p_1p_2p_3}\cup L$ is a cylinder, as well as its preimage in $Y$, which has fiber class $L$.
  The collection $C_1$ of such cylinders with varying $p$ is transversal and $H$-polar for $H\in\Ample^\circ(Y)$.

  Consider another map $\sigma'\colon Y\to\PP^2$ which contracts the (-2)-curve and two (-1)-curves $L_{12}$ and $L_{13}$, keeping $E_2$ and $E_3$. Repeating the cylinder construction, we again obtain a transversal collection $C_2$, which is $H$-polar for $H\in\Ample^\circ(Y)$.
  Then the union $C_1\cup C_2$ is $H$-complete for $H\in\Ample^\circ(Y)$. Indeed, its complement consists of curves $E_2,E_1-E_2, L_{13}$ and is not a support of any ample divisor.
\end{example}

\begin{theorem}\label{th:deg6-weak}
  Let $Y$ be a weak del Pezzo surface of degree 6.
  Then the affine cone $\AffCone_H(Y)$ is generically flexible for any very ample divisor class $H$.
\end{theorem}
\begin{proof}
  There are two non-toric weak del Pezzo surfaces of degree 6, see \cite[Section 8.4.2]{Dolg-CAG}. They are covered in Examples~\ref{ex:weak-deg6-collinear}, \ref{ex:weak-deg6-infinitesimal}.

  If $Y$ is toric, then the statement follows from \cite{AKZ}.
\end{proof}

\appendix
\section{Sagemath module usage workflow}
\label{sec:usage}

Let us describe usage scenarios of the Sagemath module \cite{Pe23-sage}.

To use the provided module, install Sagemath and put the file \texttt{delPezzo\_cylinders.py} from \cite{Pe23-sage} in your working directory.
A del Pezzo surface (i.e., its type) is created with \mintinline{python}|S=Surface(d)|, where \mintinline{python}|d| equals its degree.
 The following example code imports the module and creates a \mintinline{python}|Surface| instance.
 \begin{minted}{python}
 from delPezzo_cylinders import *
 S = Surface(3)
 \end{minted}

The implementations of examples above are the following.

\begin{itemize}
  \item Example~\ref{ex:lines} is implemented via
  \begin{minted}{python}
    Cylinder.make_type_Lines(S, S.E, S.E[i-1])
    \end{minted}    
    \item Example~\ref{ex:tangent} is implemented in the method \mintinline{python}{Cylinder.make_type_tangent}.
    \item The collection in Example~\ref{ex:tangent} is constructed via  
    \begin{minted}{python}
    Cylinder.make_type_tangent(S, S.E, S.E[k:], [S.E[k-1]], \
         [[e] for e in S.E[:k-1]])
    \end{minted}
  \item  Example~\ref{ex:cuspidal} is implemented via
  \begin{minted}{python}
  Cylinder.make_type_cuspcubic(S, S.E, S.E[-4:]) 
  \end{minted}   
  \item The code corresponding to Example~\ref{ex:cuspcubic-flex} (for $B_{5-d}$) is as follows.
  \begin{minted}{python}
  d = 4
  S = Surface(d)
  cylinder = Cylinder.make_type_cuspcubic(S, S.E, S.E[-4:])
  collection = CylinderList([cylinder])
  cone = S.cone_representative(f'B({5-d})')
  collection.is_generically_flexible_on(cone) # True
  \end{minted}
  \item   Example~\ref{ex:weak-deg6-collinear} is implemented via

  \begin{minted}{python}
S = Surface(6, collinear_triples=[[1,2,0]])
cylinder = Cylinder.make(S, 
                    complement=S.E,
                    support=S.E + [S.L-e for e in S.E],
                    fiber=S.L,
                    transversal=True
                    )
  \end{minted}
  \item Example~\ref{ex:weak-deg6-infinitesimal} is implemented in \verb|test_weak_deg6.py| in \cite{Pe23-sage}.
  \end{itemize}

  The module supports all possible kinds of (-2)-curves in the blowup model, namely, arising from collinear triple points, infinitely near points, six points on a conic, and eight points on a cuspidal cubic with a node at one of them. Thus, it supports all weak del Pezzo surfaces.

The lists of (-1)- and (-2)-curves of a constructed (weak) del Pezzo surface are available via methods
 \mintinline{python}|S.minus_one_curves()|  and \mintinline{python}|S.minus_two_curves()| respectively. The former is constructed using \cite[Lemma 8.2.2]{Dolg-CAG}, which states that a divisor class $D$ represents a (-1)-curve if and only if $D^2=-1$ and $D\cdot F\ge0$ for each (-2)-curve $F$.

The following example code chooses the subdivision cone $B_3$ and a cylinder collection from Example~\ref{ex:cuspidal}.
\begin{minted}{python}
from delPezzo_cylinders import *
S = Surface(3)
B3 = S.cone_representative('B(3)')
collection = Cylinder.make_type_cuspcubic(S, S.E, S.E[-4:])
\end{minted}   
The list of available subdivision cones for \mintinline{python}|S| is returned by 
\begin{minted}{python}
NE_SubdivisionCone.cone_types(S)
\end{minted}   
The rays of polarity and forbidden cones of a collection object named \mintinline{python}|c| are returned by \mintinline{python}|c.Pol.rays()| and \mintinline{python}|c.Forb.rays()| respectively.

The properties of the collection in the relative interior of a given cone ($B_3$ here) as well as subdivision cones, where the collection is polar and complete, can be checked with the following methods.

\begin{minted}{python}
collection.is_polar_on(B3) # False
collection.is_complete_on(B3) # True
collection.is_transversal() # True
collection.is_generically_flexible_on(B3) # False
list(collection.compatible_representatives()) # ['B(2)', 'C(2)']
list(collection.compatible_representatives(complete=True)) \
# ['B(2)', 'C(2)']
\end{minted}

The method \mintinline{python}|S.all_cylinders(constructions)| returns a collection comprised of all cylinders of certain constructions (e.g., \mintinline{python}|constructions = ['lines','tangent']|) corresponding to all choices of $\sigma$. Such a collection is useful in conjuction with the following methods.
The method \mintinline{python}|collection.make_polar_on(cone)| filters out the cylinders that are not polar inside the cone, and \mintinline{python}|collection.reduce()| removes abundant cylinders from the collection while keeping the forbidden cone unchanged.

\bibliographystyle{plainurl}
\bibliography{references} 

\begin{thebibliography}{10}

\bibitem{Arzhantsev23}
I.~Arzhantsev.
\newblock Automorphisms of algebraic varieties and infinite transitivity.
\newblock {\em St. Petersburg Mathematical Journal}, 34(2):143--178, mar 2023.
\newblock URL: \url{https://doi.org/10.1090%2Fspmj%2F1749}, \href {https://doi.org/10.1090/spmj/1749} {\path{doi:10.1090/spmj/1749}}.

\bibitem{AFKKZ}
I.~Arzhantsev, H.~Flenner, S.~Kaliman, F.~Kutzschebauch, and M.~Zaidenberg.
\newblock Flexible varieties and automorphism groups.
\newblock {\em Duke Math. J.}, 162(4):767--823, 2013.

\bibitem{AKZ}
I.V. Arzhantsev, K.~Kuyumzhiyan, and M.~Zaidenberg.
\newblock Flag varieties, toric varieties, and suspensions: three instances of infinite transitivity.
\newblock {\em Sbornik: Math}, 203(7):923–949, 2012.

\bibitem{CPW13}
Ivan Cheltsov, Jihun Park, and Joonyeong Won.
\newblock Affine cones over smooth cubic surfaces.
\newblock {\em Journal of the European Mathematical Society}, 18(7):1537--1564, 2016.
\newblock \href {https://doi.org/10.4171/jems/622} {\path{doi:10.4171/jems/622}}.

\bibitem{CPW15}
Ivan Cheltsov, Jihun Park, and Joonyeong Won.
\newblock Cylinders in del {Pezzo} surfaces.
\newblock {\em International Mathematics Research Notices}, 2017(4):1179--1230, 2017.
\newblock \href {https://doi.org/10.1093/imrn/rnw063} {\path{doi:10.1093/imrn/rnw063}}.

\bibitem{Cohen-quartic}
Teresa Cohen.
\newblock Investigations on the plane quartic.
\newblock {\em American Journal of Mathematics}, 41(3):191--211, 1919.
\newblock URL: \url{http://www.jstor.org/stable/2370332}.

\bibitem{Dolg-CAG}
Igor~V. Dolgachev.
\newblock {\em Classical Algebraic Geometry: A Modern View}.
\newblock Cambridge University Press, 2012.
\newblock \href {https://doi.org/10.1017/CBO9781139084437} {\path{doi:10.1017/CBO9781139084437}}.

\bibitem{hoff2022flexibility}
Michael Hoff and Hoang~Le Truong.
\newblock Flexibility of affine cones over mukai fourfolds of genus $g\ge7$, 2022.
\newblock \href {https://arxiv.org/abs/2208.09109} {\path{arXiv:2208.09109}}.

\bibitem{Kaliman21}
S.~Kaliman and D.~Udumyan.
\newblock On automorphisms of flexible varieties.
\newblock {\em Advances in Mathematics}, 396:108112, 2022.
\newblock URL: \url{https://www.sciencedirect.com/science/article/pii/S000187082100551X}, \href {https://doi.org/10.1016/j.aim.2021.108112} {\path{doi:10.1016/j.aim.2021.108112}}.

\bibitem{Kaliman2022embedding}
Shulim Kaliman.
\newblock {Embedding Theorems for Flexible Varieties}.
\newblock {\em Michigan Mathematical Journal}, pages 1 -- 14, 2024.
\newblock \href {https://doi.org/10.1307/mmj/20226268} {\path{doi:10.1307/mmj/20226268}}.

\bibitem{KimPark21}
Jaehyun Kim and Jihun Park.
\newblock Generic flexibility of affine cones over del {Pezzo} surfaces of degree 2.
\newblock {\em International Journal of Mathematics}, 32(14):2150104, 2021.
\newblock \href {https://doi.org/10.1142/S0129167X21501044} {\path{doi:10.1142/S0129167X21501044}}.

\bibitem{KimWon25}
Jaehyun Kim and Joonyeong Won.
\newblock Cylinders in smooth del {Pezzo} surfaces of degree 2.
\newblock {\em Advances in Geometry}, 25(1):71--91, January 2025.
\newblock URL: \url{https://www.degruyter.com/document/doi/10.1515/advgeom-2024-0034/html}, \href {https://doi.org/10.1515/advgeom-2024-0034} {\path{doi:10.1515/advgeom-2024-0034}}.

\bibitem{KPZ11}
T.~Kishimoto, Yu. Prokhorov, and M.~Zaidenberg.
\newblock Group actions on affine cones.
\newblock In {\em Affine Algebraic Geometry, CRM Proc. and Lecture Notes}, volume~54, page 123–163, Providence, RI, 2011. Amer. Math. Soc.

\bibitem{Lubbes}
Niels Lubbes.
\newblock Minimal families of curves on surfaces.
\newblock {\em Journal of Symbolic Computation}, 65:29--48, 2014.
\newblock URL: \url{https://www.sciencedirect.com/science/article/pii/S0747717114000133}, \href {https://doi.org/10.1016/j.jsc.2014.01.003} {\path{doi:10.1016/j.jsc.2014.01.003}}.

\bibitem{Manin}
Yu.~I. Manin.
\newblock {\em Cubic forms: algebra, geometry, arithmetic}.
\newblock North--Holland, Amsterdam, 1974.

\bibitem{PW16}
Jihun Park and Joonyeong Won.
\newblock Flexible affine cones over del {Pezzo} surfaces of degree 4.
\newblock {\em European Journal of Mathematics}, 2(1):304--318, Mar 2016.
\newblock \href {https://doi.org/10.1007/s40879-015-0054-4} {\path{doi:10.1007/s40879-015-0054-4}}.

\bibitem{P13}
A.~Yu. Perepechko.
\newblock Flexibility of affine cones over del {Pezzo} surfaces of degree 4 and 5.
\newblock {\em Functional Analysis and Its Applications}, 47(4):284--289, Oct 2013.
\newblock \href {https://doi.org/10.1007/s10688-013-0035-7} {\path{doi:10.1007/s10688-013-0035-7}}.

\bibitem{Pe20}
Alexander Perepechko.
\newblock Affine cones over cubic surfaces are flexible in codimension one.
\newblock {\em Forum Mathematicum}, 33(2):339--348, 2021.
\newblock \href {https://doi.org/10.1515/forum-2020-0191} {\path{doi:10.1515/forum-2020-0191}}.

\bibitem{Pe23-sage}
Alexander Perepechko.
\newblock del {P}ezzo {S}agemath module, 2023.
\newblock URL: \url{{https://github.com/aperep/delPezzo}}, \href {https://doi.org/10.5281/zenodo.7915020} {\path{doi:10.5281/zenodo.7915020}}.

\bibitem{PZ21flex}
Yuri Prokhorov and Mikhail Zaidenberg.
\newblock {\em Affine Cones over Fano--Mukai Fourfolds of Genus 10 are Flexible}, pages 363--383.
\newblock Springer International Publishing, Cham, 2023.
\newblock \href {https://doi.org/10.1007/978-3-031-11938-5_16} {\path{doi:10.1007/978-3-031-11938-5_16}}.

\bibitem{Saw24}
Masatomo Sawahara.
\newblock Polarized cylinders in du val del pezzo surfaces of degree two, 2024.
\newblock \href {https://arxiv.org/abs/2412.09848} {\path{arXiv:2412.09848}}.

\bibitem{HangTruong2023flex}
Nguyen Thi Anh~Hang and Hoang Le~Truong.
\newblock The affine cones over fano–mukai fourfolds of genus 7 are flexible.
\newblock {\em International Mathematics Research Notices}, 2024(10):8417--8426, 11 2023.
\newblock \href {https://arxiv.org/abs/https://academic.oup.com/imrn/article-pdf/2024/10/8417/57728767/rnad275.pdf} {\path{arXiv:https://academic.oup.com/imrn/article-pdf/2024/10/8417/57728767/rnad275.pdf}}, \href {https://doi.org/10.1093/imrn/rnad275} {\path{doi:10.1093/imrn/rnad275}}.

\bibitem{Won22}
Joonyeong Won.
\newblock Flexibility of affine cones over singular del {P}ezzo surfaces with degree 4.
\newblock {\em East Asian mathematical journal}, 38(3):321--329, 05 2022.

\end{thebibliography}
\end{document}